\documentclass{aims}
\usepackage{amsmath}
  \usepackage{paralist}
  \usepackage{graphics} 
  \usepackage{epsfig} 
\usepackage{graphicx}  \usepackage{epstopdf}

     \def\DOI{
     }

\newtheorem{theorem}{Theorem}[section]

\newtheorem{lemma}[theorem]{Lemma}
\newtheorem{proposition}{Proposition}

\theoremstyle{definition}
\newtheorem{definition}[theorem]{Definition}
\newtheorem{remark}{Remark}

\newcommand{\R}{\mathbb R} 
\newcommand{\N}{\mathbb{N}} 
\DeclareMathOperator*\ess{ess\,sup}

\title[Well-posedness for degenerate parabolic equations] 
      {Well-posedness for a class of nonlinear degenerate parabolic equations}

\author[Giuseppe Floridia]{}



 \email{floridia.giuseppe@icloud.com}

\thanks{The research of the author was carried out in the frame of Programme STAR, financially supported by UniNA and Compagnia di San Paolo.\\
This work was supported by the ÒInstituto Nazionale di Alta MatematicaÓ (INdAM), through the GNAMPA Research Project 2014: 
\lq\lq Controllo moltiplicativo per modelli diffusivi nonlineari'' (coordinator G. Floridia). 
Moreover, this research was 
performed in the framework of the GDRE CONEDP (European Research Group on \lq\lq Control of Partial Differential Equations'') issued by CNRS, INdAM and Universit\'e de Provence.
}

\begin{document}
\maketitle

\centerline{\scshape Giuseppe Floridia }
\medskip
{\footnotesize
 \centerline{
 University of Naples \lq\lq Federico II''
}
   \centerline{Department of Mathematics and Applications \lq\lq R. Caccioppoli'' 
%
   }
        \centerline{Via Cintia, Monte S. Angelo}
        \centerline{I-80126 Napoli, Italy}
} 



\bigskip


\begin{abstract}
In this paper we obtain well-posedness for a class of semilinear weakly degenerate reaction-diffusion systems with Robin boundary conditions. 
This result is obtained through a Gagliardo-Nirenberg interpolation inequality and some embedding results for weighted Sobolev spaces.
\end{abstract}
\section{Introduction}

In this work we study well-posedness of semilinear parabolic systems in one space dimension 
of the form
\begin{equation}
\label{Psemilineare}
\left\{\begin{array}{l}
\displaystyle{u_t-(a(x) u_x)_x =\alpha(t,x)u+ f(t,x,u)\,\quad \mbox{ in } \; Q_T \,:=\,(0,T)\times(-1,1) }\\ [2.5ex]
\displaystyle{
\begin{cases}
\begin{cases}
\beta_0 u(t,-1)+\beta_1 a(-1)u_x(t,-1)= 0 \quad \;\:t\in (0,T)\, \\
\qquad\qquad\qquad\qquad\qquad\qquad\qquad\qquad\qquad\qquad\quad\quad(\mbox{for }\, WDP)\\
\gamma_0\, u(t,1)\,+\,\gamma_1\, a(1)\,u_x(t,1)= 0
 \qquad\quad\, t\in (0,T)\,
\end{cases}
\\
\quad a(x)u_x(t,x)|_{x=\pm 1} = 0\,\,\qquad\qquad\qquad\;\;\,\, t\in(0,T)\;\;\quad(\mbox{for }\, SDP)
\end{cases}
}\\ [2.5ex]
\displaystyle{u(0,x)=u_0 (x) \,\qquad\qquad\qquad\qquad\quad\qquad\qquad\;\; \,x\in(-1,1)}~.
\end{array}\right.
\end{equation}
The equation in the \textit{Cauchy}
problem 
above is a degenerate parabolic equation because the diffusion coefficient $a\,(a\in C^0([-1,1]$)), positive on $(-1,1),$ 
vanishes at the extreme points of $[-1,1]$.
Furthermore, two kinds of degenerate diffusion coefficient can be distinguished. $(\ref{Psemilineare})$ is a weakly degenerate problem $(WDP)$ (see \cite{CF2} and \cite{PhGF}) if the diffusion coefficient
$a\in C^1(-1,1)$ and $\frac{1}{a}\in L^1(-1,1),$ while the problem $(\ref{Psemilineare})$ is called a strongly degenerate problem $(SDP)$ (see \cite{F1}, \cite{CF1} and \cite{PhGF}) if the diffusion coefficient 
$a \in C^1([-1,1]),$ consequently $\frac{1}{a}\not\in L^1(-1,1)$.\\
Some physical motivations for the study of degenerate parabolic problems 
come
 from mathematical models in climate science (see, e.g., \cite{F1
}).

\subsection{Problem formulation}
We consider the problem (\ref{Psemilineare})
under the following assumptions:
 \begin{enumerate}
  \item[(A.1)] $u_0 \in L^2(-1,1);$ 
 \item[(A.2)] $\alpha \in L^\infty (Q_T);$ 
\item[(A.3)] $f:Q_T\times\R\rightarrow \R$ is such that
\begin{itemize}
\item $(t,x,u)\longmapsto f(t,x,u)$ is a Carath\'eodory function on $Q_T\times\R,$
(\footnote{ We say that $f:Q_T\times\R\rightarrow \R$ is a Carath\'eodory function on $Q_T\times\R$ if the following properties hold:
\begin{itemize}
\item$(t,x)\longmapsto f(t,x,u)$ is measurable, for every $ u\in\R,$
 \item$u\longmapsto f(t,x,u)$ is continuous, for a.e. $(t,x)\in Q_T.$
 \end{itemize}
 })
\item
$t\longmapsto f(t,x,u)$ is loc. absolutely continuous for a.e. $x\in(-1,1),\forall u\in\R,$
\item there exist constants $\gamma_0\geq 0, \vartheta\in[1,\vartheta_{\sup})$ and $\nu
\geq0$
 such that
\begin{equation}\label{Superlinearit}
|f(t,x,u)|\leq\gamma_0\,|u|^\vartheta, \mbox{ for a.e. } (t,x)\in Q_T, \forall u\in \R\,,
\end{equation}
\begin{equation}\label{Remf}
\big(f(t,x,u)-f(t,x,v)\big)(u-v)\leq\nu
(u-v)^2, \; \mbox{ for a.e. } (t,x)\in Q_T, \forall u,v\in \R, 
\end{equation}
\begin{equation}\label{Remlip}
\big|f(t,x,u)-f(t,x,v)\big|\leq
\nu
(1+|u|^{\vartheta-1}\!\!\!+|v|^{\vartheta-1})|u-v|,\,\mbox{ for a.e.}(t,x)\!\in\! Q_T, \forall u,v\in \R,
 \end{equation}
%
\begin{equation*}
  f_t(t,x,u)\,u\geq-\nu\, 
  u^2,\; \mbox{ 
 for a.e.
 } \,(t,x)\in Q_T, \forall u\in \R;\,
\end{equation*}
\end{itemize}
\item[(A.4)] $a \in C^0([-1,1])$ is such that
    $$a(x)>0, \,\, \forall \, x \in (-1,1),\quad a(-1)=a(1)=0;$$
    \item[$({A.5}_{WD})$] 
  \begin{itemize}
  \item[$\star$] 
$\vartheta_{\sup}=4,$ 
\item[$\star$] 
 $a \in C^1(-1,1)$ is such that $\frac{1}{a}\in L^1(-1,1),$
\item[$\star$] 
   $\beta_0,\beta_1,\gamma_0,\gamma_1\in\R,\;\beta_0^2+\beta_1^2>0, \;\gamma_0^2+\gamma_1^2>0,$ satisfy the sign condition
     $\beta_0\beta_1 \leq 0 \mbox{ and } \gamma_0\gamma_1 \geq 0;$
\end{itemize}
  \item[$({A.5}_{SD})$] 
  \begin{itemize}
  \item[$\star$] 
  $\vartheta_{\sup}=3,$ 
       \item[$\star$] $a \in C^1([-1,1])$ is such that
   the function $\xi_a(x):=
   \displaystyle\int_0^x \frac{1}{a(s)}ds 
    \in L^{q_\vartheta}(-1,1),$
  where 
    $ q_\vartheta=\max\Big\{\frac{1+\vartheta}{3-\vartheta}, 2\vartheta-1\Big\}.$
  \end{itemize}
\end{enumerate}

\begin{remark}
The following is an example of function $f$ that satisfies the assumption $(A.3)\!:$ 
$\;\;\,f(t,x,u)=c(t,x)\min\{|u|^{\vartheta-1},1\}u-|u|^{\vartheta-1}u,$
where $c$ 
is a Lipschitz continuous function. 
\end{remark}
\subsection{Main results}
In this work, we are interested in the following existence and uniqueness result.
\begin{theorem}[main theorem]\label{main}
For each $u_0\in L^2(-1,1)$ there exists a unique strong solution (\footnote{See Definition \ref{strong}, for the precise definition of strong solutions.}) 
of the $WDP$ \eqref{Psemilineare} under the assumptions $(A.1)-(A.4)$ and $({A.5}_{WD})$. 
\end{theorem}
\begin{remark}
We note that Theorem \ref{main} holds also for the
$SDP$ \eqref{Psemilineare} under the assumptions $(A.1)-(A.4)$ and $({A.5}_{SD})$, with weighted Neumann boundary conditions. This result has already been obtained by the author in \cite{F1}.\\ 
\end{remark}
\subsection{Structure of this paper}
The main theorem, Theorem \ref{main}, is proved 
in Section \ref{WD}, with general Robin boundary conditions, applying the new Gagliado-Nirenberg interpolation inequalities and the embedding results for weighted Sobolev spaces obtained in Section \ref{WSE} (references to \lq\lq{\it interpolation inequalities}'' can be found in \cite{
BR}, \cite{D}, \cite{FR3} and \cite{FR2}).
 In Section \ref{C}, we show an application to the global approximate multiplicative controllability for system \eqref{Psemilineare} and we present some perspectives for this kind of controllability.

\section{Interpolation inequalities and embedding results for weighted Sobolev spaces}\label{WSE}
In this section, first we introduce some weighted Sobolev spaces and we obtain a Gagliardo-Nirenberg interpolation inequality, then we prove embedding results for spaces involving time.
\subsection{The function spaces $H^1_a(-1,1)$ and $H^2_a(-1,1)$}
In order to deal with the well-posedness of nonlinear 
$WDP$ $(\ref{Psemilineare})$, it is necessary to introduce the weighted Sobolev spaces $H^1_a(-1,1)$ and $H^2_a(-1,1)$ (see also \cite{CF2} and \cite{PhGF}).\\
We define
$$H^1_a(-1,1):=\{u\in L^2(-1,1)|\,u \in AC_{loc}([-1,1])
\text{ and }\;\sqrt{a}\,u_x\in L^2(-1,1)\},\,(\footnote{By $AC_{loc}([-1,1])$ we denote the space of the locally absolutely continuous functions on $[-1,1].$
})$$
$$H^2_a(-1,1):=\{u\in H^1_a(-1,1)| \, au_x \in H^1 (-1,1)\},$$ 
respectively, with the following norms
$$\|u\|_{1,a}^2:=\|u\|_{L^2(-1,1)}^2+|u|_{1,a}^2 \mbox{ and } \|u\|_{2,a}^2:=\|u\|_{1,a}^2+\|(au_x)_x\|^2_{L^2(-1,1)},$$
where $|u|_{1,a}^2:=\|\sqrt{a}u_x\|_{L^2(-1,1)}^2$ is a seminorm.\\
$H^1_a(-1,1)$
and $H^2_a(-1,1)$ are Hilbert spaces with their natural scalar products. 
\noindent In the following, we will sometimes use $\|\cdot\|
$ 
 instead of
$\|\cdot\|_{L^2(-1,1)}, 
$ 
and $\|\cdot\|_\infty$ instead of $\|\cdot\|_{L^\infty(Q_T)}.$  
\subsection{The operator $(A,D(A))$}
In this work we consider
the operator $(A,D(A))$
defined by
\vspace{-0.2cm}
 \begin{equation}\label{D(A)}
   \left\{\begin{array}{l}
\displaystyle{D(A)= \left\{u\in H^2_a (-1,1)\bigg\rvert
\begin{cases}
\beta_0 u(-1)+\beta_1 a(-1)u_x(-1)= 0 
\\
\gamma_0\, u(1)\,+\,\gamma_1\, a(1)\,u_x(1)= 0 
\end{cases}
 \right\}}\\ [4.5ex]
\displaystyle{A\,u=(au_x)_x+\alpha\,u, \,
\,\, \forall \,u \in D(A)}\,,
\end{array}\right.
 \end{equation}
where $\alpha\in L^\infty (-1,1).$ 
In \cite{CF2} we showed that $A$ is a closed, self-adjoint, dissipative operator with dense domain in $L^2 (-1,1)$.
Therefore, $A$ is the infinitesimal generator of a $C_0-\mbox{semigroup}$ of contractions in $L^2 (-1,1)$.

\subsection{Interpolation inequalities and embedding results for 
the space $H^1_a$}
First, we obtain the following
\begin{lemma}\label{sob1}
For every $p\geq 1,$ it follows that
\begin{equation*}
H^{1}_a (-1,1)\hookrightarrow L^{p}(-1,1)\,.
\end{equation*}
Moreover, there exists a positive constant $c, \, c=c(p),$ such that
$$\|u\|_{L^{p}(-1,1)}\leq c\,
\|u\|_{1,a},\;\;\forall u\in H^1_a(-1,1).$$
\end{lemma}
\begin{proof}
It is sufficient to prove Lemma \ref{sob1} for every $p\geq2.$ Let $u\in H^1_a(-1,1)$ and $p\geq2.$ 
Firstly, since $\frac{1}{a}\in L^1(-1,1),$ for every $x\in(-1,1),$ by H\"older's inequality we have the following estimate
\begin{equation}\label{S1.2}
\!|u(x)-u(0)|= \left|\int_0^x u'(s)\, ds\right|\,
\leq
\left|\int_0^x a(s)|u'(s)|^2 ds\right|^\frac{1}{2}\!\left|\int_0^x \frac{1}{a(s)} ds\right|^\frac{1}{2}\!\!\!\leq
k\,|u|_{1,a},
\end{equation}
where $k=\left|\int_{-1}^1 \frac{1}{a(s)} ds\right|^\frac{1}{2}.$
Furthermore, 
we obtain
\begin{equation}\label{S1.3}
\!\!|u(0)|\leq\int_{-1}^1|u(0)|\,dx\leq\int_{-1}^1|u(x)-u(0)|\,dx+\int_{-1}^1|u(x)|\,dx
\leq 2\,k\,|u|_{1,a}
+\sqrt{2}\|u\|.
\end{equation}
Finally, from (\ref{S1.2}) and (\ref{S1.3}) we deduce 
\begin{multline*}
\int_{-1}^1 |u(x)|^{p}\,dx\leq 
2^{p-1}
\int_{-1}^1 \left(|u(x)-u(0)|^{p}+|u(0)|^{p}\right)\,dx\\
\leq 
2^{p}
k^p\,|u|^{p}_{1,a}+
2^{2p}\Big(\max\Big\{k,\frac{\sqrt{2}}{2}\Big\}\Big)^{p}\|u\|^{p}_{1,a}\leq 2^{2p}\Big(\max\Big\{k,\frac{\sqrt{2}}{2}\Big\}\Big)^{p}\|u\|^{p}_{1,a}.
\end{multline*}

\end{proof}


Now, we prove the following $L^\infty$ weighted Gagliardo-Nirenberg interpolation inequality.
\begin{lemma}\label{GNl1}
\begin{equation*}
H^{1}_a (-1,1)\hookrightarrow L^{\infty}(-1,1),\,
\end{equation*}
moreover, for every $q\geq\frac{1}{2}$ there exists a positive constant $c, \, c=c(q),$ such that
\begin{equation}\label{GN1}
\|u\|_{L^{\infty}(-1,1)}\leq c\,
\|u\|^\alpha_{1,a}\,\|u\|^{1-\alpha}_{L^{2q}(-1,1)},\;\;\forall u\in H^1_a(-1,1),
\end{equation}
where $\alpha=\frac{2}{2+q}.$ 
\end{lemma}
\begin{proof}
Let us consider the auxiliary function
$F:\R\rightarrow\R,\;F(t)=|t|^{\frac{1}{\alpha}-1}t, \,t\in\R.$ We note that $F^\prime(t)=\frac{1}{\alpha}|t|^{\frac{1}{\alpha}-1}, t\in \R.$
Let us start with the particular case when $u\in H^1_a(-1,1)$ and $u(0)=0.$\\
We can deduce the following equality
$$F(u(x))=\int_0^xF^\prime(u(s))u^\prime(s)\,ds,\;\;\forall x\in(-1,1),$$
then, for every $x\in(-1,1),$ 
 \begin{equation*}
 |u(x)|^{\frac{1}{\alpha}-1}u(x)=\frac{1}{\alpha}\int_0^x|u(s)|^{\frac{1}{\alpha}-1}u^\prime(s)\,ds
 =\frac{1}{\alpha}\int_0^x\frac{1}{\sqrt[4]{a(s)}}|u(s)|^{\frac{1}{\alpha}-1}\,\sqrt[4]{a(s)}\,u^\prime(s)\,ds
 .
 \end{equation*}
 Since $\frac{1}{a}\in L^1(-1,1)$ (because $a(x)$ is a weakly degenerate diffusion coefficient) and $u\in L^{4\frac{1-\alpha}{\alpha}}(-1,1)$ owing to Lemma \ref{sob1}, through generalized H\"older's inequality (with three conjugate exponents: 4,4,2), for every $x\in(-1,1),$ we have
 \begin{multline*}
 |u(x)|^{\frac{1}{\alpha}}\leq\frac{1}{\alpha}\left|\int_0^x\frac{1}{\sqrt[4]{a(s)}}|u(s)|^{\frac{1}{\alpha}-1}\,\sqrt[4]{a(s)}\,|u^\prime(s)|\,ds\right|\\
 \leq \frac{1}{\alpha}\left|\int_0^x\frac{1}{\sqrt[4]{a(s)}}|u(s)|^{\frac{1}{\alpha}-1}\,\sqrt{a(s)}\,|u^\prime(s)|\,ds\right|\\
 \leq \frac{1}{\alpha}\!\left(\int_{-1}^1\frac{1}{a(s)}\,ds\right)^\frac{1}{4}\!\!\!\left(\int_0^x\,|u(s)|^{4\frac{1-\alpha}{\alpha}}\,ds\right)^\frac{1}{4}\!\!\!\left(\int_0^x\,a(s)\,|u^\prime(s)|^2\,ds\right)^\frac{1}{2}
\!\! \!\!\!\leq\! c \|u\|^{\frac{1-\alpha}{\alpha}}_{L^{4\frac{1-\alpha}{\alpha}}
 }\!\!|u|_{1,a}.
 \end{multline*}  
 Therefore, for every $u\in H^1_a(-1,1)$ with $u(0)=0,$ we obtain
 \begin{equation}\label{GN1.1}
 \|u\|_{L^\infty(-1,1)}\leq c\|u\|^\alpha_{1,a}\|u\|^{1-\alpha}_{L^{2q}(-1,1)},
 \end{equation}
 where $q=2\frac{1-\alpha}{\alpha}.$ Then $\alpha=\frac{2}{2+q}$ and $c=c(q).$\\
 When $u\in H^1_a(-1,1) \mbox{ and } u(0)\neq0,$ we obtain the inequality \eqref{GN1} by applying the inequality \eqref{GN1.1} to the function $\xi_\sigma\,u,$ where, for every $\sigma\in\left(0,\frac{1}{2}\right),\,$
 $\xi_\sigma\in C^\infty([-1,1])$ 
 is a symmetrical cut-off function such that:\\
 \vspace{-0.5cm}
\begin{itemize}
  \item $\xi_\sigma (-x)=\xi_\sigma (x), \qquad 
   \forall x\in [-1,1];$
  \item $0\leq \xi_\sigma(x)\leq 1,  \; 
  \forall x\in [0,1];$
   $\quad\xi_\sigma (x)=0, \; 
  \forall x\in [0,\frac{\sigma}{2}];$
   $\quad\xi_\sigma (x)=1,\; 
  \forall x\in [\sigma,1].$
\end{itemize}
Indeed, $\xi_\sigma u\in H^1_a(-1,1),\,\xi_\sigma(0)u(0)=0$ and through \eqref{GN1.1} we obtain the following inequalities
\vspace{-0.3cm}
 \begin{multline*}
 \|\xi_\sigma\,u\|_{L^\infty(-1,1)}\leq c\|\xi_\sigma\,u\|^\alpha_{1,a}\|\xi_\sigma\,u\|^{1-\alpha}_{L^{2q}(-1,1)}\leq c\|u\|^\alpha_{1,a}\|u\|^{1-\alpha}_{L^{2q}(-1,1)},\;\;\forall \sigma\in\left(0,\frac{1}{2}\right),
\end{multline*}
from which the inequality \eqref{GN1} follows.
\end{proof}
Now, applying 
Lemma \ref{GNl1}, we prove the following $L^p$ weighted Gagliardo-Nirenberg interpolation inequality.
\begin{lemma}\label{GNl2}
For every $p,q\in\R$ such that $1\leq 2q<p,$ 
there exists a positive constant $c, \, c=c(p,q),$ such that
\begin{equation}\label{GN2}
\|u\|_{L^p(-1,1)}\leq c\,
\|u\|^\beta_{1,a}\,\|u\|^{1-\beta}_{L^{2q}(-1,1)},\;\;\forall u\in H^1_a(-1,1),
\end{equation}
where $\beta=\frac{2}{p}\frac{p-2q}{q+2}\;(0<\beta<1).$ 
\end{lemma}
\begin{proof}
For every $u\in H^1_a(-1,1)$ we have
\begin{multline*}
\|u\|_{L^p(-1,1)}^p=\int_{-1}^1|u|^p\,dx=\int_{-1}^1|u|^{2q}\,|u|^{p-2q}\,dx\\
\leq \|u\|_{L^\infty(-1,1)}^{p-2q}\,\int_{-1}^1|u|^{2q}\,dx =\|u\|_{L^\infty(-1,1)}^{p-2q}\|u\|_{L^{2q}(-1,1)}^{2q}.
\end{multline*}
Applying Lemma \ref{GNl1} we deduce the following inequalities
\begin{equation*}
\|u\|_{L^p(-1,1)}^p\leq c(p,q)\|u\|^{\alpha(p-2q)}_{1,a}\|u\|^{(1-\alpha)(p-2q)}_{L^{2q}(-1,1)}\|u\|_{L^{2q}(-1,1)}^{2q}.
\end{equation*}
Bearing in mind that
$\alpha=\frac{2}{2+q},$ we obtain
$$\|u\|_{L^p(-1,1)}\leq (c(p,q))^\frac{1}{p}\|u\|^{\frac{2}{p}\frac{p-2q}{q+2}}_{1,a}\|u\|^{\frac{q}{p}\frac{p+4}{q+2}}_{L^{2q}(-1,1)}.$$
Let us set $\beta:=\frac{2}{p}\frac{p-2q}{q+2}$ and observing that $1-\beta=\frac{q}{p}\frac{p+4}{q+2},$ we conclude the proof.
\end{proof}

\subsection{Spaces involving time: ${\mathcal{B}}(Q_T)$ and ${\mathcal{H}}(Q_T)$}
Given $T>0,$ let us define the Banach spaces:
$$\mathcal{B}(Q_T):=C([0,T];L^2(-1,1))\cap L^2(0,T;H^1_a (-1,1))$$
with the norm:
$\quad
\displaystyle \|u\|^2_{\mathcal{B}(Q_T)}= \sup_{t\in
[0,T]}\|u(t,\cdot)\|^2
+2\int^T_{0}\int^1_{-1}a(x)u^2_x
dx\,dt\,,
$
$$\text{and }\;
{\mathcal{H}}(Q_T):=L^{2}(0,T;D(A)
)\cap H^{1}(0,T;L^2(-1,1))\cap C([0,T];H^{1}_a(-1,1))(\footnote{$D(A)$ is the domain of the operator defined in \eqref{D(A)}.})
$$
with the norm: 
%
$\displaystyle \|u\|^2_{\mathcal{H}(Q_T)}=
 \sup_{
 [0,T]}\left(\|u\|
 ^2+\|\sqrt{a}u_x\|
 ^2\right)+\int_0^T\left(\|u_t\|
 ^2+\|(au_x)_x\|
 ^2\right)\,dt.
 \;\;(\footnote{
 It is well known that this norm is equivalent to the Hilbert norm $$\displaystyle\||u|\|^2_{\mathcal{H}(Q_T)}=
 \int_0^T\left(\|u\|
 ^2+\|\sqrt{a}u_x\|
 ^2+\|u_t\|
 ^2+\|(au_x)_x\|
 ^2\right)\,dt.$$
 })
$
\subsection{Embedding results for the spaces ${\mathcal{B}}(Q_T)$ and ${\mathcal{H}}(Q_T)$}
Thanks to the previous weighted Gagliardo-Nirenberg interpolation inequality (Lemma \ref{GNl2}) we obtain the following embedding.
\begin{lemma}\label{sob2}
Let $T>0.$ 
We have the following embedding
\begin{equation*}
{\mathcal{B}}(Q_T)
\hookrightarrow L^{5}(Q_T)\;\;
\end{equation*}
and, 
for every $p\in[1,5],$ we obtain
$$\|u\|_{L^{p}(Q_T)}\leq c\,T^{\frac{5-p}{3p}}\|u\|_{{\mathcal{B}}(Q_T)}, \;\;\forall u\in {\mathcal{B}}(Q_T) 
,$$
where $c,\; c=c(p),$ is a positive constant.
\end{lemma}
\begin{remark}
In the following proof we will consider the  norms
$$\displaystyle
\|u\|^2_{L^{2}(0,T;H^{1}_a (-1,1))}\!\!:=\!
 \int_0^T\left(\|u\|^2+\|\sqrt{a}u_x\|^2\right)dt
\;\text{   and    } \|u\|^2_{L^\infty(0,T;L^2(-1,1))}\!\!:=\!\ess_{
[0,T]}\|u\|^2.
$$
\end{remark}

\begin{proof}{(of Lemma \ref{sob2}).}
  For every $u\in {\mathcal{B}}(Q_T),$ owing to the Gagliardo-Nirenberg interpolation inequality obtained in Lemma \ref{GNl2} we deduce
\begin{
equation*}
 \|u\|^p_{L^p(Q_T)}= \int_0^T \|u\|^p_{L^p(-1,1)}\,dt
\leq c\int_0^T \,
\|u\|^{\beta p}_{1,a}\,\|u\|^{(1-\beta)p}_{L^{2q}(-1,1)}\,dt,
\end{
equation*}
then, if we choose $q=1$ in Lemma \ref{GNl2}, we obtain $\beta=\frac{2}{p}\frac{p-2}{3}$ and the previous inequality becomes
\begin{equation}\label{S1}
\!\!\!\|u\|^p_{L^p(Q_T)}
\leq c\int_0^T \,
\|u\|^{\frac{2(p-2)}{3}}_{1,a}\,\|u\|^{\frac{p+4}{3}}_{L^{2}(-1,1)}\,dt
\leq
c\|u\|^{\frac{p+4}{3}}_{L^\infty(0,T;L^{2}(-1,1))}\!\!\!\int_0^T \!\!\!\!\|u\|^{\frac{2(p-2)}{3}}_{1,a}\!\!\!dt. 
\end{%
equation}
Since $u\in L^2(0,T;H^1_a(-1,1)),$ in \eqref{S1} 
$\frac{2(p-2)}{3}\leq2$ must occur,
 therefore we have found that $p\leq5.$ 
So, for $p=5$ the inequality (\ref{S1}) becomes
\begin{multline*}
\|u\|^5_{L^5(Q_T)}\leq c\|u\|^{3}_{L^\infty(0,T;L^{2}(-1,1))}\int_0^T \,\|u\|^{2}_{1,a}\,dt
\\= c\|u\|^{3}_{L^\infty(0,T;L^{2}(-1,1))}\|u\|^{2}_{L^2(0,T;H^{1}_a(-1,1))}
\leq c\|u\|^{5}_{{\mathcal{B}}(Q_T)}.
\end{%
multline*}
Moreover, for every $p\in(2,5),$ owing to \eqref{S1} and applying H\"older's inequality (with conjugate exponents: $\frac{3}{p-2}$ and $\frac{3}{5-p}$) we obtain
\begin{multline}\label{S2}
\|u\|^p_{L^p(Q_T)}\leq
c\|u\|^{\frac{p+4}{3}}_{L^\infty(0,T;L^{2}(-1,1))}\int_0^T \,\|u\|^{\frac{2(p-2)}{3}}_{1,a}\,dt\\
\leq c\|u\|^{\frac{p+4}{3}}_{L^\infty(0,T;L^{2}(-1,1))}\left(\int_0^Tdt\right)^{\frac{5-p}{3}}\left(\int_0^T \|u\|^{2}_{1,a}\,dt\right)^{\frac{p-2}{3}}\\
= c\,T^{\frac{5-p}{3}}\|u\|^{\frac{p+4}{3}}_{L^\infty(0,T;L^{2}(-1,1))}\|u\|^{2\frac{p-2}{3}}_{L^2(0,T;H^{1}_a(-1,1))}\leq c\, T^{\frac{5-p}{3}}\|u\|^{p}_{{\mathcal{B}}(Q_T)}.
\end{%
multline}

\end{proof}
The proof of the following Lemma \ref{sob3} has not been included in this paper as is similar to that of Lemma 3.5 of \cite{F1}. 
\begin{lemma}\label{sob3}
Let $T>0,\: \mbox{ for every } p\geq1,$ 
we have
\begin{equation*}
{\mathcal{H}}(Q_T)\subset L^{2p}(Q_T) \qquad{ and }
\end{equation*}
$\qquad\qquad\displaystyle
\|u\|_{L^{2p}(Q_T)}
\leq c\,
 T^{\frac{1}{2p}}\,\|u\|_{{\mathcal{H}}(Q_T)},
$
\quad where $c$ is a positive constant.
\end{lemma}

\section{Existence and uniqueness of solutions of semilinear WDP \eqref{Psemilineare}}\label{WD}

In this section, in order to study the semilinear $WDP$ $(\ref{Psemilineare}),$ we represent it 
 in the Hilbert space $L^2(-1,1)$ as
\begin{equation}\label{NL}
 \left\{\begin{array}{l}
\displaystyle{u^\prime(t)=A\,u(t)+\phi(u)\,,\qquad  t>0 }\\ [2.5ex]
\displaystyle{u(0)=u_0\, 
}~,
\end{array}\right.
\end{equation}
where $A$ is the operator defined in (\ref{D(A)}), 
$u_0\in L^2(-1,1),$ and, for every $u\in {\mathcal{B}(Q_T)},
$ 
the Nemytskii operator associated with the $WDP\, (\ref{Psemilineare})$ is defined as
\begin{equation}\label{phimap}
\phi(u)(t,x):=f(t,x,u(t,x)),\;\;\forall (t,x)\in Q_T. 
\end{equation}
Now, 
we will deduce the following proposition.
\begin{proposition}\label{loclip}
Let $T>0, \;1\leq\vartheta\leq4.$ 
Let $f:Q_T\times\R\rightarrow \R$ be a function that satisfies assumption $(A.3),$ then 
 $\phi:{\mathcal{B}(Q_T)}\longrightarrow L^{1+\frac{1}{\vartheta}}(Q_T)$ 
is a locally Lipschitz continuous map 
 and $\phi ({\mathcal{H}(Q_T)})\subseteq L^2(Q_T).$
 \end{proposition}

\begin{proof}
For every $u\in{\mathcal{B}(Q_T)},$ applying 
Lemma \ref{sob2} (with $p=\vartheta+1$), 
$u\in L^{1+
 {\vartheta}}(Q_T)$ and through (\ref{Superlinearit}) (see assumption (A.3)) 
 we obtain
\begin{
equation*}
\int_{Q_T}|f(t,x,u(t,x))|^{1+\frac{1}{\vartheta}}\,dx\,dt
\leq \gamma_0^{1+\frac{1}{\vartheta}}\int_{Q_T}|u|^{\vartheta(1+\frac{1}{\vartheta})}\,dx\,dt
\leq
k
\,T^{\frac{4-\vartheta}{3}}
\|u\|^{\vartheta+1}_{\mathcal{B}(Q_T)} 
\,<+\infty.
\end{equation*}
Furthermore, for every $u\in{\mathcal{H}(Q_T)},$ applying Lemma \ref{sob3}, 
$u\in L^{2\vartheta}(Q_T)$ and through (\ref{Superlinearit}) (see assumption (A.3)) we deduce
$$
\int_{Q_T}|f(t,x,u(t,x))|^2\,dx\,dt
\leq \gamma_0^2\int_{Q_T}|u|^{2\vartheta}\,dx\,dt\;
\leq 
k
\,T\,\|u\|^{2\vartheta}
_{\mathcal{H}(Q_T)}
\,<+\infty.
$$
Now, we prove that $\phi:{\mathcal{B}(Q_T)}\longrightarrow L^{1+\frac{1}{\vartheta}}(Q_T)$
is a locally Lipschitz continuous map.
Indeed, for every $u,v\in {\mathcal{B}(Q_T)},$ by \eqref{Remlip}, 
applying Lemma \ref{sob2},
we have

\begin{multline*}
\!\!\!\!\!\!\|\phi(u)-\phi(v)\|^{1+\frac{1}{\vartheta}}_{L^{1+\frac{1}{\vartheta}}(Q_T)}\!\!\!=\!\!\int_{Q_T}|f(t,x,u)-f(t,x,v)|^{1+\frac{1}{\vartheta}}\,dx\,dt\;\\
\!\!\leq\!
c\int_{Q_T}(1+|u|^{\frac{\vartheta^2-1}{\vartheta}}\!\!+|v|^{\frac{\vartheta^2-1}{\vartheta}})
|u-v|^{1+\frac{1}{\vartheta}}\!dx\,dt\;\,\\
\leq c\Big(\int_{Q_T}(1+|u|^{\vartheta+1}+|v|^{\vartheta+1})\,dx\,dt\Big)^{1-\frac{1}{\vartheta}}\;\Big(\int_{Q_T}|u-v|^{\vartheta+1}\,dx\,dt\Big)^{\frac{1}{\vartheta}}\\
\leq c\Big(T^{1-\frac{1}{\vartheta}}+\|u\|_{L^{\vartheta+1}(Q_T)}^{\frac{\vartheta^2-1}{\vartheta}}+\|v\|_{L^{\vartheta+1}(Q_T)}^{\frac{\vartheta^2-1}{\vartheta}}\Big)\|u-v\|_{L^{\vartheta+1}(Q_T)}^{1+\frac{1}{\vartheta}}\\
\leq c T^{\frac{4-\vartheta}{3\vartheta}}\Big(T^{1-\frac{1}{\vartheta}}+T^{\frac{(4-\vartheta)(\vartheta-1)}{3\vartheta}}\|u\|_{{\mathcal{B}}(Q_T)}^{\frac{\vartheta^2-1}{\vartheta}}+T^{\frac{(4-\vartheta)(\vartheta-1)}{3\vartheta}}\|v\|_{{\mathcal{B}}(Q_T)}^{\frac{\vartheta^2-1}{\vartheta}}\Big)\|u-v\|_{{\mathcal{B}}(Q_T)}^{1+\frac{1}{\vartheta}}\\
=c T^{\frac{2\vartheta+1}{3\vartheta}}\Big(1+T^{\frac{4-\vartheta}{3}}\|u\|_{{\mathcal{B}}(Q_T)}^{\frac{\vartheta^2-1}{\vartheta}}+T^{\frac{4-\vartheta}{3}}\|v\|_{{\mathcal{B}}(Q_T)}^{\frac{\vartheta^2-1}{\vartheta}}\Big)\|u-v\|_{{\mathcal{B}}(Q_T)}^{1+\frac{1}{\vartheta}}
.
\end{multline*}
Owing to the last inequalities, we prove that 
for every $R>0,$ there exists $K_R(T),$ a positive constant depending on $R$ and $T$ (increasing in $T$), such that 
 \begin{equation*}
 \|\phi(u)-\phi(v)\|_{L^{1+\frac{1}{\vartheta}}(Q_T)}\!\!\!\leq\!\!K_R(T)\|u-v\|_{\mathcal{B}(Q_T)},
 \end{equation*}
$\forall
 u,v\in{\mathcal{B}(Q_T)},\|u\|_{{\mathcal{B}(Q_T)}}\leq R,\,\|v\|_{{\mathcal{B}(Q_T)}}\!\!\leq R,$ from which we obtain the conclusion.
\end{proof}
\subsection{Strict solutions of the $WDP$ \eqref{Psemilineare}}
In order to continue, 
the next definition is necessary.
\begin{definition}
If 
$u_0\in H^1_a(-1,1),$ u is a \textit{strict solution} of the $WDP$ \eqref{Psemilineare}, 
if $u\in\mathcal{H}(Q_T)$ and
\begin{equation*}
\label{}
\left\{\begin{array}{l}
\displaystyle{u_t-(a(x) u_x)_x =\alpha(t,x)u+ \phi(u)\,\quad \mbox{ a.e. \, in } \;Q_T:=\,(0,T)\times(-1,1) }\\ [2.5ex]
\displaystyle{
\begin{cases}
\beta_0 u(t,-1)+\beta_1 a(-1)u_x(t,-1)= 0 \qquad\qquad a.e. \;\:t\in (0,T)\\
\gamma_0\, u(t,1)\,+\,\gamma_1\, a(1)\,u_x(t,1)= 0
\qquad\qquad\quad\;\;\, a.e. \;\:t\in (0,T)
\end{cases}
 }\;\;\\ [2.5ex]
\displaystyle{u(0,x)=u_0 (x) \,\qquad\qquad\qquad\qquad\quad\qquad\qquad\quad\; \,x\in(-1,1)}~.
\end{array}\right.(\footnote{
Since $u\in{\mathcal{H}}(Q_T)\subseteq L^2(0,T; D(A)
),$ we have $u(t,\cdot)\in D(A)
,$ for $ \text{ a.e. } t\in(0,T)$. 
Then, we deduce that the weighted Robin boundary conditions hold for almost every 
 $t\in(0,T)$.
})
\end{equation*}
\end{definition}
Now, we give the following existence and uniqueness result for strict solutions.
\begin{theorem}\label{exB}
For all $u_0\in H^1_a(-1,1)$ there exists a unique strict solution $u\in{\mathcal{H}(Q_T)}$ to $WDP$ \eqref{Psemilineare}.
\end{theorem}
\noindent The proof of Theorem \ref{exB} has not been included in this paper as it is very lengthy and technical. However, this proof is similar to that of the existence and uniqueness of strict solutions for the $SDP$ \eqref{Psemilineare}, proved in the Theorem 3.15 of \cite{F1} (see also Appendix B of \cite{F1}, and the author's Ph.D. Thesis \cite{PhGF}).
The fundamental idea of this proof uses the \lq\lq {\it fixed point}'' argument. 
The only difference between the SD case and the WD case consists in the fact that, in the WD case, there is the need to adapt
 the Robin boundary condition, but the sign condition on the coefficients (see $({A.5}_{WD})$) assists us.

\subsection{Strong solutions of the $WDP$ \eqref{Psemilineare}}
The following notion of \lq\lq{\it strong solutions}'' 
is classical in PDE theory, see, for instance, \cite{BDDM1}, pp. 62-64 (see also \cite{F1}).
\begin{definition}\label{strong}
Let 
$u_0\in L^2(-1,1).$ We say that $u\in\mathcal{B}(Q_T)$ is a \textit{strong solution} of the WDP \eqref{Psemilineare}, if $u(0,\cdot)=u_0$ and 
there exists a 
 sequence $\{u_k\}_{k\in\N}$ in $\mathcal{H}(Q_T)$ such that, as $k\rightarrow\infty,$ $u_k\longrightarrow u \mbox{  in }  \mathcal{B}(Q_T)$ and, for every $k\in\N$, $u_k$ is the strict solution of 
the Cauchy problem  
\begin{equation}\label{Pk}
\!\!\!\!\!\left\{\begin{array}{l}
\displaystyle{u_{kt}-(a(x) u_{kx})_x =\alpha(t,x)u_k+\phi(u_k)\,\quad \mbox{ a.e. \, in } \;Q_T:=\,(0,T)\times(-1,1), }\\ [2.5ex]
\displaystyle{
\begin{cases}
\beta_0 u_k(t,-1)+\beta_1 a(-1)u_{kx}(t,-1)= 0 \qquad\qquad a.e. \;\:t\in (0,T),\\
\gamma_0\, u_k(t,1)\,+\,\gamma_1\, a(1)\,u_{kx}(t,1)= 0
\qquad\qquad\quad\;\;\, a.e. \;\:t\in (0,T),
\end{cases}
 }
~\\ [2.5ex]
\end{array}\right.
\end{equation}
with initial datum $u_k(0,x).$
\end{definition}
\begin{remark}
We note that, due to the definition of the $\mathcal{B}(Q_T)-$norm 
(see Section 3.1), due to the fact that, as $k\rightarrow\infty,$ $u_k\longrightarrow u \mbox{  in }  \mathcal{B}(Q_T),$ from the Definition \ref{strong} we deduce that 
$\,u_k(0,\cdot)\longrightarrow u_0 \mbox{  in }  L^{2}(-1,1).$\\
Moreover, since $\phi$ is locally Lipschitz continuous (see Proposition \ref{loclip}), 
 $$\phi(u_k)\longrightarrow\phi(u), \qquad \text{ in } L^{1+\frac{1}{\vartheta}}(-1,1)
 .\,
$$ 
\end{remark}
Now, we give the following proposition.
\begin{proposition}\label{uni}
Let $T>0,  u_0, v_0\in H^1_a(-1,1)\; (\text{or } u_0, v_0\in L^2(-1,1)).$ $u,v$ are strict (or strong) solutions of the $WDP$ \eqref{Psemilineare}, with initial data $u_0, v_0$ respectively.
Then, we have
\begin{equation}\label{inedc}
\|u-v\|_{\mathcal{B}(Q_T)}\leq C_T\,\|u_0-v_0\|_{L^2(-1,1)},
\end{equation}
where $C_T=e^{(\nu+\|\alpha^+\|_{\infty})T}$ and $\alpha^+$
denotes the positive part of $\alpha$
(\footnote{
$\alpha^+(t,x):=
\max\{\alpha(t,x),0\},\;\forall (t,x)\in Q_T.$ 
}). 
\end{proposition}
The proof of Proposition \ref{uni} is similar to that of Proposition 3.16 of \cite{F1}.
\subsection{Proof of the main result 
}
Finally, we can prove the main result of this paper, that is, the existence and uniqueness of strong solutions to $WDP$ \eqref{Psemilineare} with initial data in $L^2(-1,1).$ 
\begin{proof}{(of Theorem \ref{main}).}
Let $u_0\in L^2(-1,1).$
There exists $\{u^0_k\}_{k\in\N}\subseteq H^1_a(-1,1)$ such that, as $k\rightarrow\infty,$ $u_k^0\rightarrow u_0$ in $L^2(-1,1).$
For every $k\in\N,$ we consider the problem \eqref{Pk} with initial datum $u_k(0,x)=u^0_k (x),\;x\in(-1,1).$ 
For every $k\in\N,$ through the uniqueness and existence of the strict solution to system \eqref{Pk} (see Theorem \ref{exB}), there exists a unique $u_k\in{\mathcal{H}(Q_T)}$ strict solution to \eqref{Pk}.
Then, we consider the sequence $\{u_k\}_{k\in\N}\subseteq{\mathcal{H}(Q_T)}$ and by direct application of the Proposition \ref{uni} 
 we prove that $\{u_k\}_{k\in\N}$ is a \textit{Cauchy} sequence in the Banach space $\mathcal{B}(Q_T)$. Then, there exists $u\in{\mathcal{B}(Q_T)}$ such that, as $k\rightarrow\infty,$ $u_k\rightarrow u$ in ${\mathcal{B}(Q_T)}$
and $\displaystyle u(0,\cdot)\stackrel{L^2}{=}\lim_{k\rightarrow\infty}u_k(0,\cdot)\stackrel{L^2}{=}u_0.$
So, $u\in {\mathcal{B}(Q_T)}$ is a strong solution.
 The uniqueness of the strong solution to \eqref{Psemilineare} is trivial, applying Proposition \ref{uni}.
 \end{proof}
\section{Some applications and prospectives: global multiplicative controllability}\label{C}
We are interested in studying the nonnegative multiplicative controllability of
 $(\ref{Psemilineare})$  using the \textit{bilinear control} $\alpha (t,x). 
$ Some references for multiplicative controllability are \cite{F1}, \cite{CF1}, \cite{CF2} and \cite{CK}. 
Let us start with the following definition.
\begin{definition}
We say that the system $(\ref{Psemilineare})$  is nonnegatively globally approximately controllable in $L^2(-1,1),$
if for every $
\varepsilon>0$ and for any nonnegative $
 u_0,\,u_d\in 
 L^2(-1,1),$ with $u_0\neq0$ 
there are a 
$ T=T(\varepsilon,u_0,u_d)\geq 0$
and a bilinear control 
$\alpha=
\alpha (t,x),\,\alpha
\in L^\infty(Q_T)$ such that for the corresponding strong solution 
 $u(t,x)$ of $(\ref{Psemilineare})$
we obtain
$$\|u(T,\cdot)-u_d\|_{L^2(-1,1)}\leq \varepsilon\,.$$
\end{definition}
\subsection{Multiplicative controllability for SDPs}
In the following theorem, proved in \cite{F1}, the \textit{nonnegative global approximate
controllability} result is obtained, for the semilinear
$SDP\,(\ref{Psemilineare})$ under the assumptions $(A.1),\,(A.2),\,(A.4), ({A.5}_{WD})$ and $(A.3)$ with $\vartheta\in(1,3)$ instead of $\vartheta\in[1,3)$. 

\begin{theorem}\label{T1}
The semilinear $SDP \,(\ref{Psemilineare})$  is nonnegatively globally approximately controllable in $L^2(-1,1),$ by means of 
bilinear controls $\alpha
.$
Moreover, the corresponding strong solution to the $SDP \,(\ref{Psemilineare})$  remains nonnegative a.e. in $Q_T$.
\end{theorem}
\subsection{Multiplicative controllability for WDPs}
Thanks to the well-posedness result for WDPs obtained in this paper, we will be able to investigate the possibility of extending Theorem \ref{T1} from SDPs to WDPs (see 
\cite{F3}).

\medskip
Received xxxx 20xx; revised xxxx 20xx.
\medskip

\end{document}